\documentclass [a4paper, 12pt]{article}

\usepackage{amsmath}
\usepackage{amssymb}
\usepackage{amsthm}

\usepackage{graphicx,color}

\newcommand{\sgn}{\mathrm{sgn}}
\newcommand{\const}{\mathrm{const}}
\newcommand{\inters}{\mathrm{int}}

\renewcommand{\Im}{\mathrm{Im}}
\renewcommand{\Re}{\mathrm{Re}}

\newtheorem{lemma}{Lemma}
\newtheorem{theorem}{Theorem}

\numberwithin{equation}{section}
\numberwithin{lemma}{section}
\numberwithin{theorem}{section}

\begin{document}

\begin{center}
\large A large time asymptotics for transparent potentials for the Novikov-Veselov equation at positive energy
\end{center}

\begin{center}
A. V. Kazeykina\footnote{Centre des Math\'ematiques Appliqu\'ees, Ecole Polytechnique, Palaiseau, 91128, France \\ email: kazeykina@cmap.polytechnique.fr} and R. G. Novikov\footnote{Centre des Math\'ematiques Appliqu\'ees, Ecole Polytechnique, Palaiseau, 91128, France \\ email: novikov@cmap.polytechnique.fr}
\end{center}

\textbf{Abstract.} In the present paper we begin studies on the large time asymptotic behavior for solutions of the Cauchy problem for the Novikov--Veselov equation (an analog of KdV in $ 2 + 1 $ dimensions) at positive energy. In addition, we are focused on a family of reflectionless (transparent) potentials parameterized by a function of two variables. In particular, we show that there are no isolated soliton type waves in the large time asymptotics for these solutions in contrast with well-known large time asymptotics for solutions of the KdV equation with reflectionless initial data.

\section{Introduction}
We consider the scattering problem for the two-dimensional Schr\"odinger equation
\begin{equation}
\label{schrodinger}
- \Delta \psi + v( x ) \psi = E \psi, \quad x \in \mathbb{R}^2, \quad E = E_{fixed} > 0
\end{equation}
at a fixed positive energy, where
\begin{equation}
\label{vconditions}
\begin{aligned}
& v( x ) = \overline{ v( x ) }, \quad v \in L^{ \infty }( \mathbb{R}^2 ), \\
& | v( x ) | < q ( 1 + | x | )^{ - 2 - \varepsilon }, \quad \varepsilon > 0, \; q > 0.
\end{aligned}
\end{equation}
It is known that for any $ k \in \mathbb{R}^2 $, such that $ k^2 = E $, there exists a unique bounded solution $ \psi^+( x, k ) $ of equation (\ref{schrodinger}) with the following asymptotics
\begin{equation}
\begin{aligned}
\label{scat_amplitude}
\psi^+( x, k ) = & \: e^{ i k x } - i \pi \sqrt{ 2 \pi } \, e^{ - \tfrac{ i \pi }{ 4 } } f\left( k, | k | \frac{ x }{ | x | } \right) \! \frac{ e^{ i | k | | x | } }{ \sqrt{ | k | | x | } } \; + \\
& + o\left( \frac{ 1 }{ \sqrt{ | x | } } \right), \quad | x | \to +\infty.
\end{aligned}
\end{equation}
This solution describes scattering of incident plane wave $ e^{ i k x } $ on the potential $ v( x ) $. The function $ f = f( k, l ) $, $ k \in \mathbb{R}^2 $, $ l \in \mathbb{R}^2 $, $ k^2 = l^2 = E $, arising in (\ref{scat_amplitude}), is the scattering amplitude for $ v( x ) $ in the framework of equation (\ref{schrodinger}).

In the present paper we are focused on transparent (or invisible) potentials $ v $ for equation (\ref{schrodinger}). We say that $ v $ is transparent if its scattering amplitude $ f $ is identically zero at fixed energy $ E $. We consider transparent potentials for equation (\ref{schrodinger}) as analogs of reflectionless potentials for the one--dimensional Schr\"odinger equation at all positive energies; see, for example, \cite{S}, \cite{ZM} as regards reflectionless potentials in dimension one.

In \cite{GN} it was shown that
\begin{enumerate}
\item There are no nonzero transparent potentials for equation (\ref{schrodinger}), where
\begin{equation}
\label{exp_potentials}
v( x ) = \overline{ v( x ) }, \quad v \in L^{ \infty }( \mathbb{R}^2 ), \quad | v( x ) | < \alpha e^{ \beta | x | }, \quad \alpha > 0, \, \beta > 0.
\end{equation}
\item There is a family of nonzero transparent potentials for equation (\ref{schrodinger}), where
\begin{equation}
\label{schwartz_potentials}
v( x ) = \overline{ v( x ) }, \quad v \in \mathcal{S}( \mathbb{R}^2 ),
\end{equation}
and $ \mathcal{S} $ denotes the Schwartz class.
\end{enumerate}

In the present paper, in addition to the scattering problem for (\ref{schrodinger}), we consider its isospectral deformation generated by the following $ ( 2 + 1 ) $--dimensional analog of the KdV equation:
\begin{gather}
\notag
\partial_t v = 4 \Re ( 4 \partial_z^3 v + \partial_z( v w ) - E \partial_z w ),\\
\label{MNV}
\partial_{ \bar z } w = - 3 \partial_z v, \quad v = \bar v, \\
\notag
v = v( x, t ), \quad w = w( x, t ), \quad x = ( x_1, x_2 ) \in \mathbb{R}^2, \quad t \in \mathbb{R},
\end{gather}
where
\begin{equation*}
\partial_t = \frac{ \partial }{ \partial t }, \quad \partial_z = \frac{ 1 }{ 2 } \left( \frac{ \partial }{ \partial x_1 } - i \frac{ \partial }{ \partial x_2 } \right), \quad \partial_{ \bar z } = \frac{ 1 }{ 2 } \left( \frac{ \partial }{ \partial x_1 } + i \frac{ \partial }{ \partial x_2 } \right).
\end{equation*}

Equation (\ref{MNV}) is contained implicitly in the paper of S.V. Manakov \cite{M} as an equation possessing the following representation:
\begin{equation}
\label{LAB}
\frac{ \partial ( L - E ) }{ \partial t } = [ L - E, A ] + B( L - E ),
\end{equation}
(Manakov $ L - A - B $ triple), where $ L = - \Delta + v( x, t ) $ or, in other words, $ L $ (at fixed $ t $) is the Schr\"odinger operator of (\ref{schrodinger}), $ A $ and $ B $ are suitable differential operators of the third and the zero order respectively. Equation (\ref{MNV}) was written in an explicit form by S.P. Novikov and A.P. Veselov in \cite{NV1}, \cite{NV2}, where higher analogs of (\ref{MNV}) were also constructed.

Note that both Kadomtsev--Petviashvili equations can be obtained from (\ref{MNV}) by considering an appropriate limit $ E \to \pm \infty $ (V.E. Zakharov).

In terms of scattering data the nonlinear equation (\ref{MNV}), where $ E = E_{ fixed } > 0 $,
\begin{equation}
\label{text_assumptions}
\begin{aligned}
& \text{$ v $ is sufficiently regular and has sufficient decay as $ | x | \to \infty $,}\\
& \text{$ w $ is decaying as $ | x | \to \infty $,}
\end{aligned}
\end{equation}
takes the form (\ref{b_dynamics})--(\ref{f_dynamics}) (see Section \ref{transparent_potentials}) and, in particular,
\begin{equation}
\label{scattering_dynamics}
\frac{ \partial f( k, l, t ) }{ \partial t } = 2 i \left[ k_1^3 - 3 k_1 k_2^2 - l_1^3 + 3 l_1 l_2^2 \right] f( k, l, t ),
\end{equation}
$ k = ( k_1, k_2 ) \in \mathbb{R}^2 $, $ l = ( l_1, l_2 ) \in \mathbb{R}^2 $, $ k^2 = l^2 = E $, where $ f( \cdot, \cdot, t ) $ is the scattering amplitude for $ v( \cdot, t ) $. Equation (\ref{scattering_dynamics}) implies that the nonlinear evolution equation (\ref{MNV}) under assumptions (\ref{text_assumptions}) preserves the transparency (or invisibility) property of $ v( x, 0 ) $ in the framework of the scattering problem for (\ref{schrodinger}).

In the present paper we begin studies on the large time asymptotic behavior for solutions of the Cauchy problem for (\ref{MNV}) under assumptions (\ref{text_assumptions}). We give a large time estimate for the family of solutions of (\ref{MNV}) with $ E = E_{ fixed } > 0 $ given by (\ref{mu_solve})--(\ref{v_main_formula}) and parameterized  by a function of two variables. All potentials of this family are transparent at fixed $ t $ and $ E $. In addition, this family contains all solutions of (\ref{MNV}) with $ E = E_{ fixed } > 0 $ such that:
\begin{align}
\label{schwartz}
& \bullet \, v( \cdot, 0 ) \in \mathcal{S}( \mathbb{R}^2 ), \; \text{where $ \mathcal{S} $ denotes the Schwartz class}, \\
\label{transparency}
& \bullet \, v( \cdot, 0 ) \; \text{is transparent for (\ref{schrodinger})}, \\
\label{small_norm}
& \bullet \, v( \cdot, 0 ) \; \text{satisfies (\ref{vconditions}), where}
\end{align}
\begin{equation}
\label{qQ}
q < Q( E, \varepsilon ) \quad \text{"small norm" condition},
\end{equation}
and $ Q $ is a special real function with the properties $ Q( E, \varepsilon ) > 0 $ as $ E > 0 $, $ \varepsilon > 0 $, $ Q( E, \varepsilon ) \to +\infty $ for fixed $ \varepsilon > 0 $ as $ E \to +\infty $,

\rule{24pt}{0pt}$ \bullet \; v,  w \in C^{ \infty }( \mathbb{R}^2 \times \mathbb{R} ) $, $ \partial_x^j v( x, t ) = O\left( | x |^{ -3 } \right) $, $ w( x, t ) = o( 1 ) $, $ | x | \to \infty $, for $ t \in \mathbb{R} $, $ j \in ( \mathbb{N} \bigcup 0 )^2 $.


The aforementioned family of solutions of (\ref{MNV}) was considered for the first time in \cite{GN}. In the present work we prove the following estimate
\begin{equation}
\label{main_estimate}
| v( x, t ) | \leqslant \frac{ \const ( v ) \ln( 3 + | t | ) }{ 1 + | t | }, \quad x \in \mathbb{R}^2, t \in \mathbb{R}
\end{equation}
for each $ v $ of this family.

Estimate (\ref{main_estimate}) implies that there are no isolated soliton type waves in the large time asymptotics for $ v( x, t ) $, in contrast with large time asymptotics for solutions of the KdV equation with reflectionless initial data.

Apparently, it is not difficult to obtain the estimate (\ref{main_estimate}), where the right--hand side is replaced by $ \frac{ \const( v ) }{ 1 + | t | } $ and even to give precise expression for the leading term of the asymptotics of $ v( x, t ) $ as $ t \to \infty $. However, already (\ref{main_estimate}) in its present form implies the aforementioned absence of isolated soliton--type waves in the large--time asymptotics for $ v( x, t ) $.

Studies on the large time asymptotics for solutions of the Cauchy problem for the Kadomtsev--Petviashvili equations were fulfilled in \cite{K1}, \cite{K2}, \cite{K}.

Estimate (\ref{main_estimate}) is proved in sections 3--4, using the stationary phase method, techniques developed in \cite{GN} and \cite{K} and an analysis of some cubic algebraic equation depending on a complex parameter.

\section{Transparent Potentials and Inverse Scattering Transform}
\label{transparent_potentials}
In order to study the large time behavior of the family of transparent potentials described in the previous section, we will use the inverse scattering transform for the two-dimensional Schr\"odinger equation (\ref{schrodinger}) described in \cite{GN}.

First, we give the complete definition of scattering data for (\ref{schrodinger}). Let $ k \in \mathbb{C}^2 $, $ k^2 = E $, $ \Im k \neq 0 $ and $ v( x ) $ satisfy conditions (\ref{vconditions}) and (\ref{qQ}). Then there exists a unique solution of (\ref{schrodinger}) such that
\begin{equation*}
\psi( x, k ) = e^{ i k x } ( 1 + o( 1 ) ), \quad | x | \to \infty.
\end{equation*}
It can be shown (\cite{N2}) that for $ E \in \mathbb{R} $, $ \Im k \neq 0 $ the function $ \psi( x, k ) $ can be expanded as
\begin{equation*}
\psi( k, x ) = e^{ i k x } - \pi \sgn( \Im k_2, \bar k_1 ) e^{ i k x } \left( \frac{ a( k ) }{ - k_2 x_1 + k_1 x_2 } + \frac{ e^{ -2 i \Re k x } b( k ) }{ - \bar k_2 x_1 + \bar k_1 x_2  } + o\left( \frac{ 1 }{ | x | } \right) \right).
\end{equation*}

For $ E > 0 $ the function $ b( k ) $ is considered to be scattering data for (\ref{schrodinger}) in addition to the scattering amplitude $ f( k, l ) $ arising in (\ref{scat_amplitude}). It was shown in \cite{GM,N1,N2} that at fixed positive energy $ f( k, l ) $ and $ b( k ) $ uniquely determine the potential $ v( x ) $ satisfying (\ref{vconditions}), (\ref{qQ}) (while $ f( k, l ) $ alone is insufficient for this purpose).

If potential $ v( x, t ) $ satisfies the Novikov--Veselov equation (\ref{MNV}) under assumptions (\ref{text_assumptions}), then dynamics of the scattering data is described by the following formulas
\begin{align}
\label{b_dynamics}
& \frac{ \partial b( k, t ) }{ \partial t } = 2 i \left[ k_1^3 + \bar k_1^3 - 3 k_1 k_2^2 + 3 \bar k_1 \bar k_2^2 \right] b( k, t ), \quad k \in \mathbb{C}^2, \Im k \neq 0, k^2 = E,\\
\label{f_dynamics}
& \frac{ \partial f( k, l, t ) }{ \partial t } = 2 i \left[ k_1^3 - 3 k_1 k_2^2 - l_1^3 + 3 l_1 l_2^2 \right] f( k, l, t ), \quad k, l \in \mathbb{R}^2, k^2 = l^2 = E.
\end{align}
Equation (\ref{f_dynamics}) implies that the Novikov--Veselov equation preserves the property of transparency. It was also shown in \cite{GN} that this equation does not preserve, in general, the property of very fast decay of initial data. It can only be guaranteed that if $ v( x, 0 ) \in \mathcal{S}( \mathbb{R}^2 ) $ and satisfies (\ref{vconditions}), (\ref{qQ}), then for every $ t $ we have $ v( x, t ) = O( | x |^{ -3 } ) $.

In the present paper we are concerned with transparent potentials, i.e. in the further considerations we assume that $ f( k, l, t ) \equiv 0 $. We will also put
\begin{equation}
\label{Eone}
E = 1
\end{equation}
without loss of generality (the case of an arbitrary fixed positive energy may be reduced to (\ref{Eone}) by scaling transformation). Along with the function $ \psi( x, k ) $ we will consider the function $ \mu( x, k ) $ related to $ \psi( x, k ) $ by the following expression
\begin{equation}
\label{psi_mu}
\psi( x, k ) = e^{ i k x } \mu( x, k ).
\end{equation}
It is also convenient in the two-dimensional scattering theory to introduce new notations
\begin{equation*}
z = x_1 + i x_2, \quad \bar z = x_1 - i x_2, \quad \lambda = k_1 + i k_2.
\end{equation*}
Then
\begin{equation*}
k_1 = \frac{ 1 }{ 2 } \left( \lambda + \frac{ 1 }{ \lambda } \right), \quad k_2 = \frac{ i }{ 2 } \left( \frac{ 1 }{ \lambda } - \lambda \right)
\end{equation*}
and we will consider that
\begin{equation*}
\psi = \psi( z, \lambda, t ), \quad \mu = \mu( z, \lambda, t ), \quad b = b( \lambda, t ).
\end{equation*}
In new notations the Schr\"odinger equation takes the form
\begin{equation*}
L \psi = E \psi, \quad L = - 4 \partial_{ z } \partial_{ \bar z } + v( z, t ), \quad z \in \mathbb{C},
\end{equation*}
equation (\ref{b_dynamics}) is written as
\begin{equation*}
\frac{ \partial b( \lambda, t ) }{ \partial t } = i \left( \lambda^3 + \frac{ 1 }{ \lambda^3 } + \bar \lambda^3 + \frac{ 1 }{ \bar \lambda^3 } \right) b( \lambda, t ),
\end{equation*}
and (\ref{psi_mu}) takes the form
\begin{equation*}
\psi( z, \lambda, t ) = e^{ \frac{ i }{ 2 }( \lambda \bar z + z / \lambda ) } \mu( z, \lambda, t ).
\end{equation*}
We also note that in the present paper notation $ f( z ) $ does not imply that $ f $ is holomorphic on $ z $, i.e. we omit the dependency on $ \bar z $ in the notations.

Let a transparent potential $ v( x, t ) $ satisfy at $ t = 0 $ conditions (\ref{vconditions}), (\ref{qQ}). Then the function $ \mu( z, \lambda, t ) $ has the following properties (see \cite{GN}):
\begin{enumerate}
\item $ \mu( z, \lambda, t ) $ is continuous on $ \lambda \in \mathbb{C} $;
\item $ \forall \lambda \in \mathbb{C} $, $ | \lambda | \neq 1 $, the function $ \mu( z, \lambda, t ) $ satisfies the equation
\begin{equation}
\label{bar_d}
\frac{ \partial \mu( z, \lambda, t ) }{ \partial \bar \lambda } = r( \lambda, z , t ) \overline{ \mu( z, \lambda, t ) },
\end{equation}
where
\begin{gather}
\label{r_b_connection}
r( \lambda, z, t ) = \exp( i S( \lambda, z, t ) ) r( \lambda ), \quad r( \lambda ) = \frac{ \pi \sgn ( \lambda \bar \lambda - 1 ) }{ \bar \lambda } b( \lambda, 0 ),\\
\label{s_definition}
S( \lambda, z, t ) = \left\{ -\frac{1}{2} \left( \bar \lambda z + \lambda \bar z + \frac{ z }{ \lambda } + \frac{ \bar z }{ \bar \lambda } \right) \right\} + \left\{ t \left( \lambda^3 + \bar \lambda^3 + \frac{ 1 }{ \lambda^3 } + \frac{ 1 }{ \bar \lambda^3 } \right) \right\};
\end{gather}
\item $ \mu( z, \lambda, t ) \to 1 $ as $ \lambda \to 0, \infty $.
\end{enumerate}
Properties 1--3 uniquely determine $ \mu( z, \lambda, t ) $ for all $ \lambda \in \mathbb{C} $.

Under the same assumptions on the potential and if $ v( x, 0 ) \in \mathcal{S}( \mathbb{R}^2 ) $, the function $ b( \lambda, t ) $ has the following properties: for every $ t \in \mathbb{R} $
\begin{gather}
\label{b_continuity}
b( \cdot, t ) \in \mathcal{S}( \mathbb{C} ); \\
\label{b_symmetry}
b( 1 / \bar \lambda, t ) = b( \lambda, t ), \quad b( -1 / \bar \lambda, t ) = \overline{ b( \lambda, t ) }; \\
\label{b_derivs}
\left. \partial^{ m }_{ \lambda } \partial^{ n }_{ \bar \lambda } b( \lambda, t ) \right|_{ | \lambda | = 1 } = 0 \quad \text{for all} \; m, n \geq 0.
\end{gather}

The reconstruction of the transparent potential $ v( z, t ) $ from these scattering data is based on the following scheme.
\begin{enumerate}
\item Function $ \mu( z, \lambda, t ) $ is constructed as the solution of the following integral equation
\begin{equation}
\label{mu_solve}
\mu( z, \lambda, t ) = 1 - \frac{ 1 }{ \pi } \iint_{ \mathbb{C} } r( \zeta, z , t ) \overline{ \mu( z, \zeta, t ) } \frac{ d \Re \zeta d \Im \zeta }{ \zeta - \lambda }.
\end{equation}
which is uniquely solvable if the scattering data $ b( \lambda, t ) $ satisfy properties (\ref{b_continuity})--(\ref{b_derivs}). Equation (\ref{mu_solve}) is obtained from (\ref{bar_d}) by applying the Cauchy-Green formula
\begin{equation*}
f( \lambda ) = -\frac{ 1 }{ \pi } \iint\limits_{ D } ( \partial_{ \bar \zeta } f( \zeta ) ) \frac{ d \Re \zeta d \Im \zeta }{ \zeta - \lambda } + \frac{ 1 }{ 2 \pi i } \oint\limits_{ \partial D } f( \zeta ) \frac{ d \zeta }{ \zeta - \lambda }.
\end{equation*}
\item Expanding $ \mu( z, \lambda, t ) $ as $ \lambda \to \infty $,
\begin{equation}
\label{mu_expansion}
\mu( z, \lambda, t ) = 1 + \frac{ \mu_{ -1 }( z, t ) }{ \lambda } + o\left( \frac{ 1 }{ | \lambda | } \right),
\end{equation}
we define $ v( z, t ) $ as
\begin{equation}
\label{v_main_formula}
v( z, t ) = 2 i \partial_{ z } \mu_{ -1 }( z, t ).
\end{equation}
\item It can be shown (\cite{GM}) that
\begin{equation*}
L \psi = \psi
\end{equation*}
where
\begin{gather*}
\psi( z, \lambda, t ) = e^{ \frac{ i }{ 2 } ( \lambda \bar z + z / \lambda ) } \mu( z, \lambda, t ), \quad L = - 4 \partial_{ z } \partial_{ \bar z } + v( z, t ), \\
\overline{ v( z, t ) } = v( z, t ), \quad v( z, t ) \; \text{is transparent}.
\end{gather*}
\end{enumerate}

\section{Estimate for the linearized case}
Consider
\begin{equation}
\label{lin_solution}
\begin{aligned}
& I( t, z ) = \iint\limits_{ \mathbb{C} } f( \zeta ) \exp( i S( \zeta, z, t ) ) d \Re \zeta d \Im \zeta, \\
& J( t, z ) = - 3 \iint\limits_{ \mathbb{C} } \frac{ \bar \zeta }{ \zeta } f( \zeta ) \exp( i S( \zeta, z, t ) ) d \Re \zeta d \Im \zeta,
\end{aligned}
\end{equation}
where $ f( \zeta ) \in L^1( \mathbb{C} ) $, $ S $ is defined by (\ref{s_definition}). If $ v( z, t ) = I( t, z ) $, $ w( z, t ) = J( t, z ) $, where
\begin{equation*}
( | \zeta |^3 + | \zeta |^{ -3 } ) f( \zeta ) \in L^{ 1 }( \mathbb{C} )
\end{equation*}
as a function of $ \zeta $, and, in addition,
\begin{equation*}
\overline{ f( \zeta ) } = f( - \zeta ) \quad \text{and/or} \quad \overline{f( \zeta )} = -| \zeta |^{-4} f\left( - \frac{ 1 }{ \bar \zeta } \right),
\end{equation*}
then $ v $, $ w $ satisfy the linearized Novikov--Veselov equation (\ref{MNV}) with $ E = 1 $. In addition,
\begin{equation*}
\hat v( p, t ) \equiv 0 \quad \text{for} \quad | p | < 2, \quad t \in \mathbb{R},
\end{equation*}
where $ \hat v( \cdot, t ) $ is the Fourier transform of $ v( \cdot, t ) $, that is $ v( \cdot, t ) $ is transparent in the Born approximation at energy $ E = 1 $ for each $ t \in \mathbb{R} $.

The goal of this section is to give, in particular, a uniform estimate of the large--time behavior of the integral $ I( t, z ) $ of (\ref{lin_solution}) under the assumptions that
\begin{equation}
\label{f_assumptions}
\begin{aligned}
& f \in C^{ \infty }( \mathbb{C} ), \\
& \partial^{ m }_{ \lambda } \partial^{ n }_{ \bar \lambda } f( \lambda ) =
\begin{cases}
& O\left( | \lambda |^{ -\infty } \right) \quad \text{ as } | \lambda | \to \infty, \\
& O\left( | \lambda |^{ \infty } \right) \quad \text{ as } | \lambda | \to 0,
\end{cases}\\
& \partial^{ m }_{ \lambda } \partial^{ n }_{ \bar \lambda } f( \lambda )|_{ | \lambda | = 1 } = 0
\end{aligned}
\end{equation}
for all $ m, n \geqslant 0 $.

Applying the classical stationary phase method to (\ref{lin_solution}), (\ref{f_assumptions}) (see, for example, \cite{F}) yields
\begin{equation}
\label{simple_estimate}
| I( t, z ) | = O\left( \frac{ 1 }{ | t | } \right), t \to \infty,
\end{equation}
uniformly on $ z \in K $, where $ K $ is any compact set of the complex plane. This is not sufficient to guarantee the absence of soliton--type waves in the large time asymptotics of the potential $ v( z, t ) = I( t, z ) $. So our further reasoning will be devoted to obtaining an estimate like (\ref{simple_estimate}) uniformly on $ z \in \mathbb{C} $.

For this purpose we introduce parameter $ u = \frac{ z }{ t } $ and write the integral $ I $ in the following form
\begin{equation}
\label{new_form}
I( t, u ) = \iint\limits_{ \mathbb{C} } f( \zeta ) \exp( i t S( u, \zeta ) ) d \Re \zeta d \Im \zeta,
\end{equation}
where
\begin{equation}
\label{s_function}
S( u, \zeta ) = -\frac{ 1 }{ 2 } \left( \bar \zeta u + \zeta \bar u + \frac{ u }{ \zeta } + \frac{ \bar u }{ \bar \zeta } \right) + \left( \zeta^3 + \bar \zeta^3 + \frac{ 1 }{ \zeta^3 } + \frac{ 1 }{ \bar \zeta^3 } \right).
\end{equation}

We will start by studying the properties of the stationary points of the function $ S( u, \zeta ) $. These points satisfy the equation
\begin{equation}
\label{stationary}
S'_{ \zeta } = - \frac{ \bar u }{ 2 } + \frac{ u }{ 2 \zeta^2 } + 3 \zeta^2 - \frac{ 3 }{ \zeta^4 } = 0.
\end{equation}

The degenerate stationary points obey additionally the equation
\begin{equation}
\label{degenerate}
S''_{ \zeta \zeta } = - \frac{ u }{ \zeta^3 } + 6 \zeta + \frac{ 12 }{ \zeta^5 } = 0.
\end{equation}

We denote $ \xi = \zeta^2 $ and
\begin{equation*}
Q( u, \xi ) = - \frac{ \bar u }{ 2 } + \frac{ u }{ 2 \xi } + 3 \xi - \frac{ 3 }{ \xi^2 }.
\end{equation*}
For each $ \xi $, a root of the function $ Q( u, \xi ) $, there are two corresponding stationary points of $ S( u, \zeta ) $, $ \zeta = \pm \sqrt{ \xi } $.

The function $ S'_{ \zeta }( u, \zeta ) $ can be represented in the following form
\begin{equation}
\label{sprime_representation}
S'_{ \zeta }( u, \zeta ) = \frac{ 3 }{ \zeta^4 } ( \zeta^2 - \zeta_0^2( u ) ) ( \zeta^2 - \zeta_1^2( u ) ) ( \zeta^2 - \zeta_2^2( u ) ).
\end{equation}

We will also use hereafter the following notations:
\begin{equation*}
\mathcal{U} = \{ u = 6 ( 2 e^{ - i \varphi } + e^{ 2 i \varphi } ), \; \varphi \in [ 0, 2 \pi ) \}
\end{equation*}
and
\begin{equation*}
\mathbb{U} = \{ u = r e^{ i \varphi } \colon r \leq | 6 ( 2 e^{ - i \varphi } + e^{ 2 i \varphi } ) |, \; \varphi \in [ 0, 2 \pi ) \},
\end{equation*}
the domain limited by the curve $ \mathcal{U} $ (see also Figure \ref{u_curve}).

\begin{lemma}
\label{lin_lemma}
\rule{1pt}{0pt}

\begin{enumerate}
\item If $ u = 18 e^{ \frac{ 2 \pi i k }{ 3 } } $, $ k = 0, 1, 2 $, then
\begin{equation*}
\zeta_0( u ) = \zeta_1( u ) = \zeta_2( u ) = e^{ -\frac{ \pi i k }{ 3 } }
\end{equation*}
and $ S( u, \zeta ) $ has two degenerate stationary points, corresponding to a third-order root of the function $ Q( u, \xi ) $, $ \xi_1 = e^{ -\frac{ 2 \pi i k }{ 3 } } $.

\item If $ u \in \mathcal{U} $ $($ i.e. $ u = 6( 2 e^{ - i \varphi } + e^{ 2 i \varphi } ) $ $)$ and $ u \neq 18 e^{ \frac{ 2 \pi i k }{ 3 } } $, $ k = 0, 1, 2 $, then \begin{equation*}
\zeta_0( u ) = \zeta_1( u ) = e^{ i \varphi / 2 }, \quad \zeta_2( u ) = e^{ - i \varphi }.
\end{equation*}

Thus $ S( u, \zeta ) $ has two degenerate stationary points, corresponding to a second-order root of the function $ Q( u, \xi ) $, $ \xi_1 = e^{ i \varphi } $, and two non--degenerate stationary points corresponding to a first-order root, $ \xi_2 = e^{ - 2 i \varphi } $.

\item If $ u \in \inters \mathbb{U} $, then
\begin{equation*}
\zeta_i( u ) = e^{ i \varphi_i }, \quad \text{and} \quad \zeta_i( u ) \neq \zeta_j( u ) \quad \text{for} \quad i \neq j.
\end{equation*}
In this case the stationary points of $ S( u, \zeta ) $ are non-degenerate and correspond to the roots of the function $ Q( u, \xi ) $ with absolute values equal to 1.

\item If $ u \in \mathbb{C} \backslash \mathbb{U} $, then
\begin{equation*}
\zeta_0( u ) = ( 1 + \omega ) e^{ i \varphi / 2 }, \quad \zeta_1( u ) = e^{ - i \varphi }, \quad  \zeta_2( u ) = ( 1 + \omega )^{ -1 } e^{ i \varphi / 2 }
\end{equation*}
for certain $ \varphi $ and $ \omega > 0 $.

In this case the stationary points of the function $ S( u, \zeta ) $ are non-degenerate, and correspond to the roots of the function $ Q( u, \xi ) $ that can be expressed as $ \xi_0 = ( 1 + \tau ) e^{ i \varphi } $, $ \xi_1 = e^{ - 2 i \varphi } $, $ \xi_2 = ( 1 + \tau )^{ - 1 } e^{ i \varphi } $, $ ( 1 + \tau ) = ( 1 + \omega )^2 $.

\end{enumerate}
\end{lemma}

Lemma \ref{lin_lemma} is proved in section \ref{proof_section}.

Formula (\ref{sprime_representation}) and Lemma \ref{lin_lemma} give a complete description of the stationary points of the function $ S( u, \zeta ) $.

In order to estimate the large--time behavior of the integral having the form
\begin{equation}
\label{more_complex_integral}
I( t, u, \lambda ) = \iint\limits_{ \mathbb{C} } f( \zeta, \lambda ) \exp( i t S( u, \zeta ) ) d \Re \zeta d \Im \zeta
\end{equation}
uniformly on $ u, \lambda \in \mathbb{C} $, in the present and the following sections we will use the following general scheme.
\begin{enumerate}
\item Consider $ D_{ \varepsilon } $, the union of disks with a radius of $ \varepsilon $ and centers in singular points of function $ f( \zeta, \lambda ) $ and stationary points of $ S( u, \zeta ) $.
\item Represent $ I( t, u, \lambda ) $ as the sum of integrals over $ D_{ \varepsilon } $ and $ \mathbb{C} \backslash D_{ \varepsilon } $:
\begin{equation}
\label{int_sum}
\begin{aligned}
& I( t, u, \lambda ) = I_{ int } + I_{ ext }, \quad \text{ where } \\
& I_{ int } = \iint\limits_{ D_{ \varepsilon } } f( \zeta, \lambda ) \exp( i t S( u, \zeta ) ) d \Re \zeta d \Im \zeta, \\
& I_{ ext } = \iint\limits_{ \mathbb{ C } \backslash D_{ \varepsilon } } f( \zeta, \lambda ) \exp( i t S( u, \zeta ) ) d \Re \zeta d \Im \zeta.
\end{aligned}
\end{equation}
\item Find an estimate of the form
\begin{equation*}
| I_{ int } | = O\left( \varepsilon^{ \alpha } \right), \quad \text{as} \quad \varepsilon \to 0 \quad ( \alpha \geq 1 )
\end{equation*}
uniformly on $ u $, $ \lambda $, $ t $.
\item Integrate $ I_{ ext } $ by parts using Stokes formula
\begin{multline}
\label{ext_by_parts}
I_{ ext } =  \frac{ 1 }{ 2 t } \int\limits_{ \partial D_{ \varepsilon } } \frac{ f( \zeta, \lambda ) \exp( i t S( u, \zeta ) ) }{ S'_{ \zeta }( u, \zeta ) } d \bar \zeta
- \frac{ 1 }{ i t } \iint\limits_{ \mathbb{C} \backslash D_{ \varepsilon } } \frac{ f'_{ \zeta }( \zeta, \lambda ) \exp( i t S( u, \zeta ) ) }{ S'_{ \zeta }( u, \zeta ) } d \Re \zeta d \Im \zeta - \\
- \frac{ 1 }{ i t } \iint\limits_{ \mathbb{C} \backslash D_{ \varepsilon } } \frac{ f( \zeta, \lambda ) \exp( i t S( u, \zeta ) ) S''_{ \zeta \zeta }( u, \zeta ) }{ ( S'_{ \zeta }( u, \zeta ) )^2 } d \Re \zeta d \Im \zeta = \\
= - \frac{ 1 }{ t } ( I_1 - I_2 - I_3 ).
\end{multline}
\item For each $ I_i $ find an estimate of the form
\begin{equation*}
(a) \: | I_i | = O\left( \ln \frac{ 1 }{ \varepsilon } \right) \quad \text{or} \quad (b) \: | I_i | = O\left( \frac{ 1 }{ \varepsilon^{ \beta } } \right), \quad \text{ as } \varepsilon \to 0.
\end{equation*}
\item In case $ (a) $ set $ \varepsilon = \dfrac{ 1 }{ | t | } $ which yields the overall estimate
\begin{equation*}
| I( t, u, \lambda ) | = O\left( \frac{ \ln( | t | ) }{ | t | } \right), \quad \text{as} \quad t \to \infty.
\end{equation*}
In case $ (b) $ set $ \varepsilon = \dfrac{ 1 }{ | t |^{ k } } $, where $ k( \alpha + \beta ) = 1 $, which yields the overall estimate
\begin{equation*}
| I( t, u, \lambda ) | = O\left( \dfrac{ 1 }{ | t |^{ \frac{ \alpha }{ \alpha + \beta } } } \right), \quad \text{as} \quad t \to \infty.
\end{equation*}
\end{enumerate}

Using this scheme we obtain, in particular, the following result
\begin{lemma}
\label{lin_estimate_lemma}
Under assumptions (\ref{f_assumptions}), (\ref{new_form}),
\begin{equation*}
| I( t, u ) | = O\left( \frac{ \ln( 3 + | t | ) }{ 1 + | t | } \right) \quad \text{for} \quad t \in \mathbb{R}
\end{equation*}
uniformly on $ u \in \mathbb{C} $.
\end{lemma}
A detailed proof of Lemma \ref{lin_estimate_lemma} is given in section \ref{proof_section}.

\section{Estimate for the non--linearized case}
In this section we prove estimate (\ref{main_estimate}) for the solution $ v( x, t ) $ of the Cauchy problem for the Novikov--Veselov equation at positive energy with the initial data $ v( x, 0 ) $ satisfying properties (\ref{schwartz})--(\ref{qQ}) or, more generally, for $ v( x, t ) $ constructed by means of (\ref{r_b_connection})--(\ref{v_main_formula}).

We proceed from the formulas (\ref{mu_expansion}), (\ref{v_main_formula}) for the potential $ v( z, t ) $ and the integral equation (\ref{mu_solve}) for $ \mu( z, \lambda, t ) $.

We write (\ref{mu_solve}) as
\begin{equation}
\label{mu_int}
\mu( z, \lambda, t ) = 1 + ( A_{ z, t } \mu )( z, \lambda, t ),
\end{equation}
where
\begin{equation*}
( A_{ z, t } f )( \lambda ) = \partial_{ \bar \lambda }^{ -1 }( r( \lambda ) \exp( i t S( u, \lambda ) ) \overline{ f( \lambda ) } ) = - \frac{ 1 }{ \pi } \iint\limits_{ \mathbb{C} } \frac{ r( \zeta ) \exp( i t S( u, \zeta ) ) \overline{ f( \zeta ) } }{ \zeta - \lambda } d \Re \zeta d \Im \zeta
\end{equation*}
and $ S( u, \zeta ) $ is defined by (\ref{s_function}), $ u = \dfrac{ z }{ t } $.

Equation (\ref{mu_int}) can be also written in the form
\begin{equation}
\label{mu_an_form}
\mu( z, \lambda, t ) = 1 + A_{ z, t } \cdot 1 + ( A^2_{ z, t } \mu )( z, \lambda, t ).
\end{equation}
According to the theory of the generalized analytic functions (see \cite{V}), equations (\ref{mu_int}), (\ref{mu_an_form}) have a unique solution for all $ z, t $. This solution can be written as
\begin{equation}
\label{mu_sol}
\mu( z, \lambda, t ) = ( I - A_{ z, t }^2 )^{ -1 }( 1 + A_{ z, t } \cdot 1 ).
\end{equation}
Equation (\ref{mu_sol}) possesses a formal asymptotic expansion
\begin{equation}
\label{expanded}
\mu( z, \lambda, t ) = ( I + A_{ z, t }^2 + A_{ z, t }^4 + \ldots )( 1 + A_{ z, t } \cdot 1 ).
\end{equation}
From estimate (\ref{member_estimate}) given below it follows that (\ref{expanded}) uniformly converges for sufficiently large $ t $. We will also write formula (\ref{expanded}) in the form
\begin{equation}
\label{rest_member_form}
\mu( z, \lambda, t ) = 1 + A_{ z, t } \cdot 1 + R,
\end{equation}
where $ R = \left( \sum\limits_{ k = 1 }^{ \infty }A_{ z, t }^{ 2 k } \right)( 1 + A_{ z, t } \cdot 1 ) $.

In addition to $ A_{ z, t } $ we introduce another integral operator $ B_{ z, t } $ defined as
\begin{equation}
\label{b_operator}
B_{ z, t } \cdot f = \iint\limits_{ \mathbb{C} } r( \zeta ) \exp( i t S( u, \zeta ) ) \overline{ f( \zeta ) } d \Re \zeta d \Im \zeta.
\end{equation}
To study (\ref{expanded}) we will need some estimates on the values of operators $ A_{ z, t } $ and $ B_{ z, t } $.

\begin{lemma}
\label{estimates_lemma}
Under assumptions (\ref{b_continuity})--(\ref{b_derivs}), the following estimates hold:
\begin{itemize}
\item[(a)]
\begin{equation}
\label{b_estimate}
| B_{ z, t } \cdot 1 | = O \left( \frac{ \ln( | t | ) }{ | t | } \right), \quad \text{as} \quad t \to \infty
\end{equation}
uniformly on $ u \in \mathbb{C} $;

\item[(b)]
\begin{equation}
\label{ba_estimate}
| B_{ z, t } \cdot A_{ z, t } \cdot 1 | = O \left( \frac{ \ln( | t | ) }{ | t | } \right), \quad \text{as} \quad t \to \infty
\end{equation}
uniformly on $ u \in \mathbb{C} $;

\item[(c)]
\begin{equation}
\label{a2_estimate}
| ( A_{ z, t }^2 \cdot 1 )( \lambda ) | \leqslant \frac{ \beta }{ | t |^{ 1 / 14 } }, \quad | t | \geqslant 1,
\end{equation}
where $ \beta $ is a constant independent of $ u $ and $ \lambda $;

\item[(d)]
\begin{equation}
\label{member_estimate}
| ( A_{ z, t }^n \cdot 1 )( \lambda ) | \leqslant \frac{ \beta^{ n - 1 } }{ | t |^{ \lfloor n / 2\rfloor / 14 } }, \quad | t | \geqslant 1;
\end{equation}

\item[(e)]
\begin{equation}
\label{bmember_estimate}
| B_{ z, t } \cdot A^{ n - 1 }_{ z, t } \cdot 1 | \leqslant \frac{ \beta^{ n - 1 } \ln( 3 + | t | ) }{ | t |^{ 1 + \lfloor ( n - 2 ) / 2 \rfloor / 14 } }, \quad | t | \geqslant 1.
\end{equation}
\end{itemize}
\end{lemma}

\begin{proof}[Proof of Lemma \ref{estimates_lemma}]
We proceed according to the scheme described in the previous section.
\begin{itemize}

\item[\emph{(a)}] This point follows from Lemma \ref{lin_estimate_lemma}.

\item[\emph{(b)}] As $ ( A_{ z, t } \cdot 1 )( \lambda ) \in C( \mathbb{C} ) $, we take $ D_{ \varepsilon } $ to be the union of disks of a radius $ \varepsilon $ centered in the stationary points of $ S( u, \zeta ) $. We note that $ ( A_{ z, t } \cdot 1 )( \lambda ) = O( 1 ) $ uniformly on $ u $ (or, equivalently, on $ z $) and $ \lambda $. Thus the integral $ I_{ int } $ (as in (\ref{int_sum})) can be estimated as
\begin{equation*}
I_{ int } = \iint\limits_{ D_{ \varepsilon } } r( \zeta, \bar \zeta ) \exp( i t S( u, \zeta, \bar \zeta ) ) ( A_{ z, t } \cdot 1 )( \zeta ) d \Re \zeta d \Im \zeta = O( \varepsilon^2 ).
\end{equation*}
Now let us estimate the integral $ I_{ ext } $ (as in (\ref{int_sum})). For this purpose we apply the Stokes formula (as in (\ref{ext_by_parts})) taking into consideration that $ \partial_{ \bar \lambda }( A_{ z, t } \cdot f )( \lambda ) = r( \lambda ) \exp( i t S( u, \lambda ) ) \overline{ f( \lambda ) } $:
\begin{multline}
\label{b_by_parts}
\iint\limits_{ \mathbb{C} \backslash D_{ \varepsilon } } r( \zeta ) \exp( i t S( u, \zeta ) ) ( A_{ z, t } \cdot 1 )( \zeta ) d \Re \zeta d \Im \zeta = \\
= \int\limits_{ \partial D_{ \varepsilon } } \frac{ r( \zeta ) \exp( i t S( u, \zeta ) ) ( A_{ z, t } \cdot 1 )( \zeta ) }{ 2 t S'_{ \bar \zeta }( u, \zeta ) } d \zeta - \\
- \iint\limits_{ \mathbb{C} \backslash D_{ \varepsilon } } \frac{ r'_{ \bar \zeta }( \zeta ) \exp( i t S( u, \zeta ) ) ( A_{ z, t } \cdot 1 )( \zeta ) }{ i t S'_{ \bar \zeta }( u, \zeta ) } d \Re \zeta d \Im \zeta + \\
+ \iint\limits_{ \mathbb{C} \backslash D_{ \varepsilon } } \frac{ r( \zeta ) \exp( i t S( u, \zeta ) ) ( A_{ z, t } \cdot 1 )( \zeta ) S''_{ \bar \zeta \bar \zeta }( u, \zeta ) }{ i t ( S'_{ \bar \zeta }( u, \zeta ) )^2 } d \Re \zeta d \Im \zeta - \\
- \iint\limits_{ \mathbb{C} \backslash D_{ \varepsilon } } \frac{ r^2( \zeta ) ( \exp( i t S( u, \zeta ) ) )^2 }{ i t S'_{ \bar \zeta }( u, \zeta ) } d \Re \zeta d \Im \zeta.
\end{multline}
Now, proceeding as in the proof of Lemma \ref{lin_estimate_lemma}, we obtain (\ref{ba_estimate}).

\item[\emph{(c)}]
In this case we build $ D_{ \varepsilon } $ as the union of disks with a radius of $ \varepsilon $ and centers in $ \lambda $ and stationary points of $ S( u, \zeta ) $. The integral $ I_{ int } $ over $ D_{ \varepsilon } $ behaves asymptotically as $ O( \varepsilon ) $. When estimating the integral $ I_{ ext } $ over $ \mathbb{C} \backslash D_{ \varepsilon } $ we use (\ref{sprime_representation}), (\ref{ext_by_parts}) and the following inequalities
\begin{equation*}
| \zeta - \lambda | \geqslant \varepsilon, \quad | \zeta - \zeta_i | \geqslant \varepsilon
\end{equation*}
($ \zeta_i $ are stationary points of $ S( u, \zeta ) $) which hold for all $ \zeta \in \mathbb{C} \backslash D_{ \varepsilon } $.
Thus we obtain that the asymptotical behavior of $ I_{ ext } $ is at most $ O\left( \dfrac{ 1 }{ | t | \varepsilon^{ 13 } } \right) $. Then, as proposed by the scheme, we choose $ \varepsilon = | t |^{ -1/14 } $ and obtain the required estimate.

\item[\emph{(d)}] This point is proved by induction. As in point \emph{(c)} $ D_{ \varepsilon } $ is the union of disks with a radius of $ \varepsilon $ and centers in $ \lambda $ and stationary points of $ S( u, \zeta ) $.

For the integral $ I_{ int } $ we have
\begin{multline}
\label{d_int}
| I_{ int } | = \left| \iint\limits_{ D_{ \varepsilon } } \frac{ r( \zeta ) \exp( i t S( u, \zeta ) ) }{ \zeta - \lambda } ( A^{ n - 1 }_{ z, t } \cdot 1 )( \zeta ) d \Re \zeta d \Im \zeta \right| \leqslant \\
\leqslant \frac{ \beta^{ n - 2 } }{ | t |^{ \lfloor ( n - 1 ) / 2 \rfloor / 14 } } \iint\limits_{ D_{ \varepsilon } } \left| \frac{ r( \zeta ) \exp( i t S( u, \zeta ) ) }{ \zeta - \lambda } \right| d \Re \zeta d \Im \zeta \leqslant \frac{ \beta^{ n - 1 } \varepsilon }{ | t |^{ \lfloor ( n - 2 ) / 2 \rfloor / 14 } }.
\end{multline}

To estimate $ I_{ ext } $ we use the following representation
\begin{multline}
\label{d_by_parts}
\iint\limits_{ \mathbb{C} \backslash D_{ \varepsilon } } \frac{ r( \zeta ) \exp( i t S( u, \zeta ) ) ( A_{ z, t }^{ n - 1 } \cdot 1 )( \zeta ) }{ \zeta - \lambda } d \Re \zeta d \Im \zeta = \\
= \int\limits_{ \partial D_{ \varepsilon } } \frac{ r( \zeta ) \exp( i t S( u, \zeta ) ) ( A_{ z, t }^{ n - 1 } \cdot 1 )( \zeta ) }{ 2 t ( \zeta - \lambda ) S'_{ \bar \zeta }( u, \zeta ) } d \zeta - \\
- \iint\limits_{ \mathbb{C} \backslash D_{ \varepsilon } } \frac{ r'_{ \bar \zeta }( \zeta ) \exp( i t S( u, \zeta ) ) ( A_{ z, t }^{ n - 1 } \cdot 1 )( \zeta ) }{ i t ( \zeta - \lambda ) S'_{ \bar \zeta }( u, \zeta ) } d \Re \zeta d \Im \zeta + \\
+ \iint\limits_{ \mathbb{C} \backslash D_{ \varepsilon } } \frac{ r( \zeta ) \exp( i t S( u, \zeta ) ) ( A_{ z, t }^{ n - 1 } \cdot 1 )( \zeta ) S''_{ \bar \zeta \bar \zeta }( u, \zeta ) }{ i t ( \zeta - \lambda )( S'_{ \bar \zeta }( u, \zeta ) )^2 } d \Re \zeta d \Im \zeta - \\
- \iint\limits_{ \mathbb{C} \backslash D_{ \varepsilon } } \frac{ r^2( \zeta ) ( \exp( i t S( u, \zeta ) ) )^2 ( A_{ z, t }^{ n - 2 } \cdot 1 )( \zeta ) }{ i t ( \zeta - \lambda ) S'_{ \bar \zeta }( u, \zeta ) } d \Re \zeta d \Im \zeta = \\
= J_1 + J_2 + J_3 + J_4.
\end{multline}

The integrals $ J_i $ can be estimated in the following way
\begin{equation*}
| J_1 | \leqslant \frac{ \beta^{ n - 2 } }{ | t | \cdot | t |^{ \lfloor ( n - 1 ) / 2 \rfloor / 14 } }\frac{ 1 }{ \varepsilon^7 } \int\limits_{ \partial D_{ \varepsilon } } | r( \zeta ) | d \zeta \leqslant \frac{ \beta^{ n - 1 } }{ | t | \cdot | t |^{ \lfloor ( n - 1 ) / 2 \rfloor / 14 } }\frac{ 1 }{ \varepsilon^7 }.
\end{equation*}
Similarly,
\begin{equation*}
\begin{aligned}
& | J_2 | \leqslant \frac{ \beta^{ n - 1 } }{ | t | \cdot | t |^{ \lfloor ( n - 1 ) / 2 \rfloor / 14 } }\frac{ 1 }{ \varepsilon^7 },\\
& | J_3 | \leqslant \frac{ \beta^{ n - 1 } }{ | t | \cdot | t |^{ \lfloor ( n - 1 ) / 2 \rfloor / 14 } }\frac{ 1 }{ \varepsilon^{13} },\\
& | J_4 | \leqslant \frac{ \beta^{ n - 1 } }{ | t | \cdot | t |^{ \lfloor ( n - 2 ) / 2 \rfloor / 14 } }\frac{ 1 }{ \varepsilon^7 }.
\end{aligned}
\end{equation*}
Thus,
\begin{equation*}
| I_{ ext } | \leqslant \frac{ \beta^{ n - 1 } }{ | t | \cdot | t |^{ \lfloor ( n - 2 ) / 2 \rfloor / 14 } \cdot \varepsilon^{ 13 } }.
\end{equation*}

Now we set $ \varepsilon = | t |^{ -1 / 14 } $ and obtain the overall estimate
\begin{equation*}
| ( A^n_{ z, t } \cdot 1 )( \lambda ) | \leqslant \frac{ \beta^{ n - 1 } }{ | t |^{ 1 / 14 } | t |^{ \lfloor ( n - 2 ) / 2 \rfloor / 14  } } = \frac{ \beta^{ n - 1 } }{ t^{ \lfloor n / 2 \rfloor / 14 } }.
\end{equation*}

\item[\emph{(e)}] Proceeding from \emph{(d)} this point is proved similarly to \emph{(b)}. \qedhere
\end{itemize}
\end{proof}

\begin{lemma}
\label{a_r_estimates}
Under the assumptions of Lemma \ref{estimates_lemma}, we have that:

\begin{itemize}

\item[(a)] $ A_{ z, t } \cdot 1 = \frac{ a_1( z, t ) }{ \lambda } + O \left( \frac{ 1 }{ | \lambda |^2 } \right) $ for $ \lambda \to \infty $, where
\begin{equation}
\label{a1_estimate}
| \partial_{ z } a_1( z, t ) | = O \left( \frac{ \ln( | t | ) }{ | t | } \right), \quad \text{as} \quad t \to \infty
\end{equation}
uniformly on $ z \in \mathbb{C} $.

\item[(b)] $ A^2_{ z, t } \cdot 1 = \frac{ a_2( z, t ) }{ \lambda } + O \left( \frac{ 1 }{ | \lambda |^2 } \right) $ for $ \lambda \to \infty $, where
\begin{equation}
\label{a2_estimate}
| \partial_{ z } a_2( z, t ) | = O \left( \frac{ \ln( | t | ) }{ | t | } \right), \quad \text{as} \quad t \to \infty
\end{equation}
uniformly on $ z \in \mathbb{C} $.

\item[(c)]
$ R = \frac{ q( z, t ) }{ \lambda } + O \left( \frac{ 1 }{ | \lambda |^2 } \right) $ as $ \lambda \to \infty $, where
\begin{equation}
\label{q_estimate}
| \partial_{ z } q( z, t ) | = O \left( \frac{ \ln( | t | ) }{ | t | } \right), \quad \text{as} \quad t \to \infty
\end{equation}
uniformly on $ z \in \mathbb{C} $.
\end{itemize}
\end{lemma}

\begin{proof}[Proof of Lemma \ref{a_r_estimates}]
The asymptotics for $ A_{ z, t } \cdot 1 $, $ A_{ z, t }^2 \cdot 1 $ and $ R $ follow from the definitions of $ A_{ z, t } $ and $ R $, formula (\ref{mu_expansion}) and properties (\ref{b_continuity}), (\ref{b_symmetry}). The rest of the proof consists in the following.

\begin{itemize}
\item[\emph{(a)}]
Estimate (\ref{a1_estimate}) follows from (\ref{b_estimate}) and the formula
\begin{equation*}
a_1( z, t ) = \frac{ 1 }{ \pi } \iint\limits_{ \mathbb{C} } r( \zeta  ) \exp( i t S( u, \zeta ) ) d \Re \zeta d \Im \zeta = B_{ z, t } \cdot 1.
\end{equation*}

\item[\emph{(b)}]
Estimate (\ref{a2_estimate}) follows from (\ref{ba_estimate}) and the formula
\begin{equation*}
a_2( z, t ) = \frac{ 1 }{ \pi } \iint\limits_{ \mathbb{C} } r( \zeta, \bar \zeta ) \exp( i t S( u, \zeta, \bar \zeta ) ) ( A_{ z, t } \cdot 1 )( \zeta ) d \Re \zeta d \Im \zeta = B_{ z, t } \cdot A_{ z, t } \cdot 1.
\end{equation*}

\item[\emph{(c)}]
We note that $ q( z, t ) = \sum\limits_{ k = 2 }^{ \infty } a_k( z, t ) $ where $ a_k( z, t ) $ is defined
\begin{equation*}
( A_{ z, t }^{ k } \cdot 1 )( \lambda ) = \frac{ a_k( z, t ) }{ \lambda } + O\left( \frac{ 1 }{ | \lambda |^2 } \right).
\end{equation*}

Next,
\begin{equation*}
a_k( z, t ) = \frac{ 1 }{ \pi }\iint\limits_{ \mathbb{C} } r( \zeta  ) \exp( i t S( u, \zeta ) ) ( A_{ z, t }^{ k - 1 } \cdot 1 )( \zeta ) d \Re \zeta d \Im \zeta.
\end{equation*}

Thus, proceeding as in point \emph{(e)} of lemma \ref{estimates_lemma}, we obtain that $ a_2( z, t ) + a_3( z, t ) = O\left( \frac{ \ln( | t | ) }{ | t | } \right)$ and the rest of the members form a geometric progression that converges to the sum of order $ O\left( \frac{ \ln( | t | ) }{ | t |^{ 1 + 1/14 } } \right) $. \qedhere
\end{itemize}
\end{proof}

Formulas (\ref{mu_expansion}), (\ref{v_main_formula}), (\ref{rest_member_form}) and Lemma \ref{a_r_estimates} imply
\begin{theorem}
Let $ v( x, t ) $ be a solution to the Cauchy problem for the Novikov--Veselov equation (\ref{MNV}) with $ E = 1 $, constructed via (\ref{mu_solve})--(\ref{v_main_formula}) under assumptions (\ref{r_b_connection})--(\ref{b_derivs}). Then
\begin{equation*}
| v( x, t ) | \leqslant \frac{ \const ( v ) \ln( 3 + | t | ) }{ 1 + | t | }, \quad x \in \mathbb{R}^2, t \in \mathbb{R}.
\end{equation*}
\end{theorem}

\section{Proofs of Lemmas \ref{lin_lemma} and \ref{lin_estimate_lemma}}
\label{proof_section}
\begin{proof}[Proof of Lemma \ref{lin_lemma}]
Under the additional assumption that $ \xi \neq 0 $ the system of equations (\ref{stationary})--(\ref{degenerate}) is equivalent to the following system
\begin{equation}
\label{xi_system}
\begin{cases}
& \xi^3 - \frac{ \bar u }{ 6 } \xi^2 + \frac{ u }{ 6 } \xi - 1 = 0, \\
& \xi^3 - \frac{ u }{ 6 } \xi + 2 = 0.
\end{cases}
\end{equation}

We claim that $ \xi = \xi_1 $ corresponds to a degenerate stationary point of (\ref{s_function}), iff
\begin{equation}
\label{poly_representation}
\begin{aligned}
& \text{the polynomial $ P( \xi ) = \xi^3 - \frac{ \bar u }{ 6 } \xi^2 + \frac{ u }{ 6 } \xi - 1 $ can be represented } \\
& \text{in the form $ P( \xi ) = ( \xi - \xi_1 )^2( \xi - \xi_2 ) $ }.
\end{aligned}
\end{equation}

Indeed, if $ \xi = \xi_1 \neq 0 $ is a zero of the function $ Q( u, \xi ) = - \frac{ \bar u }{ 2 } + \frac{ u }{ 2 \xi } + 3 \xi - \frac{ 3 }{ \xi^2 } = 3 P( \xi ) / \xi^2 $, then $ Q( u, \xi ) $ is holomorphic in a certain neighborhood of $ \xi = \xi_1 $ and can be expanded into the following Taylor series
\begin{equation*}
Q( u, \xi ) = c_1 ( \xi - \xi_1 ) + c_2( \xi - \xi_1 )^2 + c_3( \xi - \xi_1 )^3 + \ldots
\end{equation*}

Thus $ S'_{ \zeta } $ can be represented as
\begin{equation*}
S'_{ \zeta }( u, \zeta ) = c_1( \zeta^2 - \xi_1 ) + c_2( \zeta^2 - \xi_1 )^2 + c_3( \zeta^2 - \xi_1 )^3 + \ldots
\end{equation*}
After differentiating with respect to $ \zeta $ we obtain
\begin{equation*}
S''_{ \zeta \zeta }( u, \zeta ) = 2 c_1 \zeta  + 4 c_2 ( \zeta^2 - \xi_1 ) \zeta + 6 c_3 ( \zeta^2 - \xi_1 )^2 \zeta + \ldots
\end{equation*}
The stationary point corresponding to $ \xi = \xi_1 $ can be degenerate if and only if $ c_1 = 0 $.

So for the polynomial $ P( \xi ) $ we get the following representation in the neighborhood of $ \xi = \xi_1 $
\begin{equation*}
P( \xi ) = c_2( \xi - \xi_1 )^2 \xi^2 + O( | \xi - \xi_1 |^3 ).
\end{equation*}
As $ P( \xi ) $ is a third--order polynomial, it follows that $ P( \xi ) = ( \xi - \xi_1 )^2 ( \xi - \xi_2 ) $.

Expanding the expression for $ P( \xi ) $ and equating coefficients for the corresponding powers of $ \xi $ results in the following system
\begin{equation}
\label{xi1_xi2_system}
\begin{cases}
& 2 \xi_1 + \xi_2 = \frac{ \bar u }{ 6 }, \\
& 2 \xi_1 \xi_2 + \xi_1^2 = \frac{ u }{ 6 }, \\
& \xi_1^2 \xi_2 = 1.
\end{cases}
\end{equation}

Excluding $ u $ from the system yields
\begin{equation}
\label{xi_system}
\begin{cases}
& 2 \xi_1 \xi_2 + \xi_1^2 = 2 \bar \xi_1 + \bar \xi_2, \\
& \xi_1^2 \xi_2 = 1.
\end{cases}
\end{equation}

We represent $ \xi_1 $ in the form $ \xi_1 = r e^{ i \varphi } $. Then from the second equation in (\ref{xi_system}) we get that $ \xi_2 = r^{ -2 } e^{ - 2 i \varphi } $. Substituting this into the first equation of (\ref{xi_system}) yields
\begin{equation}
\label{key_equation}
( r^2 - r^{-2} ) e^{ 2 i \varphi } = ( r - r^{ -1 } ) e^{ - i \varphi }.
\end{equation}

For $ r = 1 $ equation (\ref{key_equation}) holds for all $ \varphi $. In addition, if $ r \neq 1 $, $ r > 0 $, then equation (\ref{key_equation}) can be rewritten as
\begin{equation*}
( r + r^{ -1 } ) e^{ 3 i \varphi } = 2
\end{equation*}
and has no solutions for real $ \varphi $.

Using now the second equation in (\ref{xi1_xi2_system}) one can see that the set of $ u $ values, for which (\ref{xi_system}), (\ref{xi1_xi2_system}) are solvable and (\ref{poly_representation}) holds, is a curve on the complex plane described in the parametric form by
\begin{equation*}
u = 6 ( 2 e^{ - i \varphi } + e^{ 2 i \varphi } )
\end{equation*}
(see Figure \ref{u_curve}). This curve has three singular points corresponding to $ \varphi = \frac{ 2 \pi k }{ 3 } $, $ k = 0, 1, 2 $. For these values of $ \varphi $ we have that $ e^{ i \varphi } = \xi_1 = \xi_2 = e^{ - 2 i \varphi } $, and $ P( \xi ) $ can be represented in the form $ P( \xi ) = ( \xi - \xi_1 )^3 $. For $ \varphi \neq \frac{ 2 \pi k }{ 3 } $, $ k = 0, 1, 2 $, we have that $ P( \xi ) = ( \xi - \xi_1 )^2 ( \xi - \xi_2 ) $ where $ e^{ i \varphi } = \xi_1 \neq \xi_2 = e^{ - 2 i \varphi } $. The first two statements of the Lemma \ref{lin_lemma} are proved.

\begin{figure}[!h]
\begin{center}
\includegraphics[width=120mm]{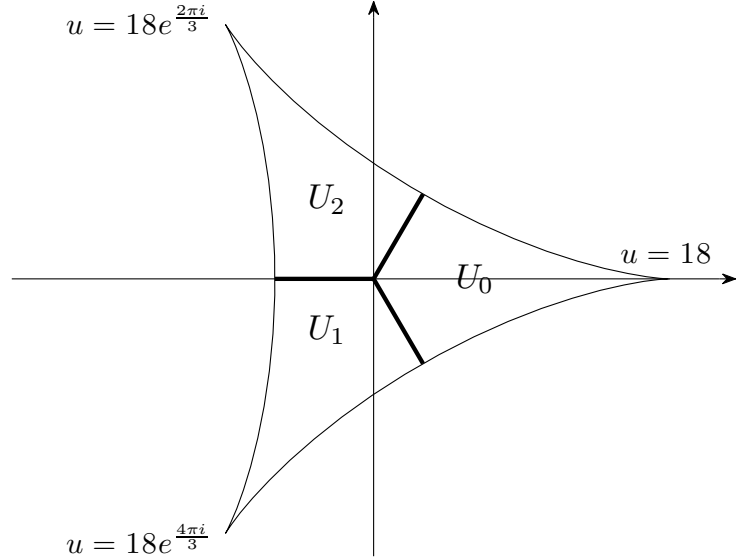}
\caption{The curve $ \mathcal{U} $ on the complex plane.}
\label{u_curve}
\end{center}
\end{figure}

Let us now fix $ \xi = e^{ i \varphi } $ and find the set of $ u $ for which this $ \xi $ is a root of the polynomial $ P( \xi ) $.

One can see that $ \xi $ is the root of $ P( \xi ) $ iff
\begin{equation}
\label{inhom_equation}
\xi \frac{ u }{ 6 } - \xi^2 \frac{ \bar u }{ 6 } = 1 - \xi^3.
\end{equation}
We now solve the homogeneous equation
\begin{equation}
\label{hom_equation}
e^{ i \varphi } \frac{ u }{ 6 } - e^{ 2 i \varphi } \frac{ \bar u }{ 6 } = 0
\end{equation}
with respect to $ u $ to find the plausible perturbations of $ u $ for which $ \xi = e^{ i \varphi } $ remains a root of $ P( \xi ) $.

From (\ref{hom_equation}) we get $ \frac{ u }{ \bar u } = e^{ i \varphi } $, and thus $ u = s e^{ i \varphi / 2 } $ for $ s \in \mathbb{R} $. So for all $ u $ that belong to the line
\begin{equation}
\label{line_equal_roots}
u( \varphi, s ) = 6 ( 2 e^{ - i \varphi } + e^{ 2 i \varphi } ) + s e^{ i \varphi / 2 }, \quad s \in \mathbb{R}
\end{equation}
one of the roots of $ P( \xi ) $ is equal to $ e^{ i \varphi } $.

Now we note that the tangent vector to $ \mathcal{ U } $
\begin{equation*}
\frac{ d }{ d \varphi } ( 2 e^{ - i \varphi } + e^{ 2 i \varphi } ) = 2 i ( e^{ 2 i \varphi } - e^{ - i \varphi } )
\end{equation*}
is collinear to the perturbation vector $ s e^{ i \varphi / 2 } $ for all $ \varphi \neq \frac{ 2 \pi i k }{ 3 } $, $ k = 0, 1, 2 $. Thus $ u( \varphi, s ) $ given by (\ref{line_equal_roots}) is the tangent line to $ \mathcal{U} $ passing through the point $ 6 ( 2 e^{ - i \varphi } + e^{ 2 i \varphi } ) $.

We note that for each $ u \in \inters \mathbb{U} $ there exist two different tangents to the curve $ \mathcal{U} $ passing through this $ u $. Indeed, note that the tangent lines to the curve $ \mathcal{U} $ passing through the points $ u = 18 e^{ \frac{ 2 \pi i k }{ 3 } } $, $ k = 0, 1, 2 $, divide the domain $ \inters \mathbb{U} $ into three parts
\begin{equation*}
\begin{aligned}
& U_0 = \{ u = r e^{ i \varphi }, \; 0 \leqslant r < 6 ( 2 e^{ - i \varphi } + e^{ 2 i \varphi } ), \; \frac{ - \pi }{ 3 } \leq \varphi \leq \frac{ \pi }{ 3 } \},\\
& U_1 = \{ u = r e^{ i \varphi }, \; 0 < r < 6 ( 2 e^{ - i \varphi } + e^{ 2 i \varphi } ), \; \frac{ \pi }{ 3 } < \varphi \leq \pi \},\\
& U_2 = \{ u = r e^{ i \varphi }, \; 0 < r < 6 ( 2 e^{ - i \varphi } + e^{ 2 i \varphi } ), \; \pi < \varphi < \frac{ 5 \pi }{ 3 } \}
\end{aligned}
\end{equation*}
(see Figure \ref{u_curve}). We first study $ U_0 $. Let us consider $ K_{ 01 } $ and $ K_{ 02 } $ which are the sets of points of the following pencils of tangent lines:
\begin{equation*}
\begin{aligned}
& K_{ 01 } = \left\{ u \in \mathbb{C} \colon u = 6 ( 2 e^{ - i \varphi } + e^{ 2 i \varphi } ) + s e^{ i \varphi / 2 }, \quad s \in \mathbb{R}, \quad - \frac{ 2 \pi }{ 3 } \leqslant \varphi < 0 \right\}, \\
& K_{ 02 } = \left\{ u \in \mathbb{C} \colon u = 6 ( 2 e^{ - i \varphi } + e^{ 2 i \varphi } ) + s e^{ i \varphi / 2 }, \quad s \in \mathbb{R}, \quad 0 < \varphi \leqslant \frac{ 2 \pi }{ 3 } \right\}.
\end{aligned}
\end{equation*}
It is easily seen that $ U_0 \subset K_{ 01 } $, $ U_0 \subset K_{ 02 } $, i.e. each point from $ U_0 $ is covered by a certain tangent line from both pencils $ K_{ 01 } $ and $ K_{ 02 } $. It can be shown similarly that each point of $ U_j $, $ j = 1, 2 $, is covered twice by the corresponding tangent lines.

Thus every point from $ \inters \mathbb{U} $ is covered twice which means that for each $ u $ in the domain limited by the curve $ \mathcal{U} $ there exist two different roots of $ P( \xi ) $ equal to $ 1 $ in absolute value. As the product of the roots of the polynomial $ P( \xi ) $ is equal to $ 1 $, the third root of the polynomial is also equal to $ 1 $ in absolute value. We do not consider the values $ u $ from the boundary of $ \mathbb{U} $ which means that the above--mentioned roots correspond to non--degenerate stationary points. Thus the third statement of Lemma \ref{lin_lemma} is proved.

Now suppose $ u \in \mathbb{C} \backslash \mathbb{U} $. We note that every such point $ u $ belongs to one and only one tangent line to the curve $ \mathcal{U} $. That means that for every $ u \in \mathbb{C} \backslash \mathbb{U} $ one of the roots of the polynomial $ P( \xi ) $, $ \xi_1 $, is such that $ | \xi_1 | = 1 $. As $ u \not \in \mathcal{U} $, $ \xi_0 \neq \xi_1 $, $ \xi_2 \neq \xi_1 $. Besides $ | \xi_0 | \neq 1 $ because otherwise there would be two different tangent lines passing through the corresponding point $ u $. That means that for every $ u \in \mathbb{C} \backslash \mathbb{U} $ there exists a root of the polynomial $ P( \xi ) $, namely $ \xi = \xi_0 $, such that $ | \xi_0 | \neq 1 $.

Considering equation (\ref{inhom_equation}) and its conjugate yields the following system of linear equations for $ u $ and $ \bar u $
\begin{equation*}
\begin{cases}
& a u + b \bar u = c, \\
& \bar b \bar u + \bar a \bar u = \bar c,
\end{cases}
\end{equation*}
where $ a = \frac{ \xi }{ 6 } $, $ b = -\frac{ \xi^2 }{ 6 } $, $ c = 1 - \xi^3 $ for each $ \xi \in \mathbb{C} $, $ | \xi | \neq 1 $, $ \xi \neq 0 $. Thus
\begin{equation*}
u = \frac{ c \bar a - \bar c b }{ a \bar a - b \bar b } =  6 \frac{ \bar \xi - \bar \xi \xi^3 + \xi^2 - \xi^2 \bar \xi^3  }{ \xi \bar \xi ( 1 - \xi \bar \xi ) } = 6 \left( \frac{ 1 }{ \xi } + \bar \xi + \frac{ \xi }{ \bar \xi } \right)
\end{equation*}
for each $ \xi \in \mathbb{C} $, $ | \xi | \neq 1 $, $ \xi \neq 0 $.

Now let us consider $ \xi = \xi_0 = ( 1 + \tau ) e^{ i \varphi } $, $ 0 < \tau < +\infty $. Then the corresponding value of the parameter $ u $ is
\begin{equation*}
u = 6 \left( 2 e^{ - i \varphi } + e^{ 2 i \varphi } + \frac{ \tau^2 }{ 1 + \tau } e^{ - i \varphi } \right).
\end{equation*}
In addition to $ \xi = \xi_0 $, for this value of the parameter $ u $ polynomial $ P( \xi ) $ has also a root $ \xi = \xi_1 = e^{ - 2 i \varphi } $. Indeed, if $ u = 6( 2 e^{ - i \varphi } + e^{ 2 i \varphi } ) $, then $ \xi_1 = e^{ - 2 i \varphi } $ is a root of the polynomial $ P( \xi ) $, and the plausible perturbation of this $ u $ for which $ \xi_1 $ remains a root of $ P( \xi ) $ is equal to $ s e^{ - i \varphi } $.

In addition, as the product of the roots of $ P $ is equal to $ 1 $, the third root is $ \xi_2 = ( 1 + \tau )^{ -1 } e^{ i \varphi } $. The fourth statement of Lemma \ref{lin_lemma} is proved.
\end{proof}

\begin{proof}[Proof of lemma \ref{lin_estimate_lemma}]
In this case $ D_{ \varepsilon } $ is the union of disks with a radius of $ \varepsilon $ centered in the stationary points of $ S( u, \zeta ) $. The integral $ I_{int} $ (as in (\ref{int_sum})) is estimated as
\begin{equation*}
| I_{ int } | = \left| \iint\limits_{ D_{ \varepsilon } } r( \zeta ) \exp( i t S( u, \zeta ) ) d \Re \zeta d \Im \zeta \right| \leqslant \const \left| \iint\limits_{ D_{ \varepsilon } } d \Re \zeta d \Im \zeta \right| = O( \varepsilon^2 ).
\end{equation*}

The estimate for $ I_{ ext } $ (as in (\ref{int_sum})) is proved separately for $ u \in \mathbb{U} $ and $ u \in \mathbb{C} \backslash \mathbb{U} $.
\begin{itemize}
\item[I.] $ u \in \mathbb{U} $:

In this case all $ \zeta \in D_{ \varepsilon } $ lie in the $ \varepsilon $--neighborhood of the unit circle. Consequently, from (\ref{r_b_connection}), (\ref{b_continuity}) and (\ref{b_derivs}) it follows that for any $ N \in \mathbb{N} $ and any $ \varepsilon_0 \in ( 0, 1/2 ] $ there exists $ C = C( \varepsilon_0, N ) $ such that
\begin{equation*}
| r( \zeta ) | \leqslant C \rho^N, \quad | r'_{ \zeta }( \zeta ) | \leqslant C \rho^N
\end{equation*}
for all $ \zeta \in D_{ \rho } $, $ \rho \leqslant \varepsilon_0 $.

The function $ S'_{ \zeta }( u, \zeta ) $ can be estimated as
\begin{align*}
& | S'_{ \zeta }( u, \zeta ) | \geqslant 3 \frac{ \varepsilon_0^6 }{ | \zeta |^4 } \quad \text{for} \: \zeta \in \mathbb{C} \backslash D_{ \varepsilon_0 }, \quad \text{and} \\
& | S'_{ \zeta }( u, \zeta ) | \geqslant 3 \frac{ \rho^6 }{ | \zeta |^4 } \quad \text{for} \: \zeta \in \partial D_{ \rho }, \quad \varepsilon \leqslant \rho \leqslant \varepsilon_0.
\end{align*}

Taking $ N = 5 $ results in the following estimate for $ I_1 $
\begin{multline*}
| I_1 | \leqslant \frac{ 1 }{ 2 } \int\limits_{ \partial D_{ \varepsilon } } \frac{ | r( \zeta ) | }{ | S'_{ \zeta }( u, \zeta ) | } d\bar \zeta \leqslant \const \frac{ \varepsilon^N }{ \varepsilon^6 } \int\limits_{ \partial D_{ \varepsilon } } | \zeta |^4 d \bar \zeta \leqslant \\
\leqslant \const \frac{ \varepsilon^N \varepsilon }{ \varepsilon^6 } ( 1 + \varepsilon )^4 = O( 1 ), \quad \text{as} \: \varepsilon \to 0.
\end{multline*}

When estimating $ I_2 $ and $ I_3 $ we integrate separately over $ D_{ \varepsilon_0 } \backslash D_{ \varepsilon } $ and $ \mathbb{C} \backslash D_{ \varepsilon } $:
\begin{multline*}
| I_2 | \leqslant \iint\limits_{ D_{ \varepsilon_0 } \backslash D_{ \varepsilon } } \left| \frac{ r'_{ \zeta }( \zeta ) \exp( i t S( u, \zeta ) ) }{ S'_{ \zeta }( u, \zeta ) } \right| d \Re \zeta d \Im \zeta + \iint\limits_{ \mathbb{C} \backslash D_{ \varepsilon_0 } } \left| \frac{ r'_{ \zeta }( \zeta ) \exp( i t S( u, \zeta ) ) }{ S'_{ \zeta }( u, \zeta ) } \right| d \Re \zeta d \Im \zeta \leqslant \\
\leqslant \const \int\limits_{ \varepsilon }^{ \varepsilon_0 } \frac{ \rho^N \rho }{ \rho^6 } d \rho + \const \iint\limits_{ \mathbb{C} \backslash D_{ \varepsilon_0 } } | r( \zeta ) || \zeta^4 | d \Re \zeta d \Im \zeta = O( 1 ), \quad \text{as} \quad \varepsilon \to 0,
\end{multline*}
\begin{multline*}
| I_3 | \leqslant \iint\limits_{ D_{ \varepsilon_0 } \backslash D_{ \varepsilon } } \left| \frac{ r( \zeta ) \exp( i t S( u, \zeta ) ) S''_{ \zeta \zeta }( u, \zeta ) }{ ( S'_{ \zeta }( u, \zeta ) )^2 } \right| d \Re \zeta d \Im \zeta  + \\
+ \iint\limits_{ \mathbb{C} \backslash D_{ \varepsilon_0 } } \left| \frac{ r( \zeta ) \exp( i t S( u, \zeta ) ) S''_{ \zeta \zeta }( u,\zeta ) }{ ( S'_{ \zeta }( u, \zeta ) )^2 } \right| d \Re \zeta d \Im \zeta \leqslant \\
\leqslant \const \int\limits_{ \varepsilon }^{ \varepsilon_0 } \frac{ \rho^N \rho }{ \rho^{12} } d \rho + \const \iint\limits_{ \mathbb{C} \backslash D_{ \varepsilon_0 } } | r( \zeta ) || \zeta^3 | d \Re \zeta d \Im \zeta \stackrel{ N = 11 }{ = } O( 1 ), \quad \text{as} \quad \varepsilon \to 0.
\end{multline*}

Setting finally $ \varepsilon = \frac{ 1 }{ | t | } $ yields
\begin{equation*}
I( t, u ) = O\left( \frac{ 1 }{ | t | } \right),\quad \text{as} \quad t \to \infty
\end{equation*}
uniformly on $ u \in \mathbb{U} $.

\item[II.] $ u \in \mathbb{C} \backslash \mathbb{U} $:

Let us divide the complex plane into six sets, each containing one and only one stationary point of $ S( u, \zeta ) $: $ \mathbb{C} = \bigcup\limits_{ k = 0 }^{ 2 } \left( Z_k^+ \bigcup Z_k^- \right) $. We define the set $ Z_k^{ \pm } $ as the set of points of the complex plane to which the stationary point $ \pm \zeta_{k} $ is the closest:
\begin{equation*}
\begin{aligned}
& Z_k^+ = \{ \zeta \in \mathbb{C} \colon | \zeta - \zeta_i | \geqslant | \zeta - \zeta_k |, \quad | \zeta + \zeta_j | \geqslant | \zeta - \zeta_k |, \quad i, j \in \{ 0, 1, 2 \} \}, \\
& Z_k^- = \{ \zeta \in \mathbb{C} \colon | \zeta - \zeta_i | \geqslant | \zeta + \zeta_k |, \quad | \zeta + \zeta_j | \geqslant | \zeta + \zeta_k |, \quad i, j \in \{ 0, 1, 2 \} \}
\end{aligned}
\end{equation*}
where $ k \in \{ 0, 1, 2 \} $. We will estimate the integral over each $ \mathcal{Z}_k ^{+}$ separately. The integrals over $ Z_{k}^- $ are treated similarly.

%
%
%

Let us first take $ \zeta \in Z_1^+ $.  Using the definition of $ Z_1^+ $ and the property that all $ \zeta \in D_{ \varepsilon } \bigcap Z_1^+ $ lie in the $ \varepsilon $--neighborhood of the unit circle one can see that the scheme of reasoning for the case I is applicable.

Now let us consider $ \zeta \in Z_0^+ \bigcup Z_2^+ $.

\begin{itemize}
\item[(A)] First, we will study the set of values of parameter $ \omega $ for which $ \zeta_2( u ) $, $ - \zeta_2( u ) $ lie outside the $ 2 \varepsilon_0 $--neighborhood of zero, i.e. $ \frac{ 1 }{ 1 + \omega } > 2 \varepsilon_0 $.

We will consider $ \zeta \in Z_0^+ $ (the case $ \zeta \in Z_2^+ $ is treated similarly). If $ \zeta \in Z_0^+ \bigcap D_{ \varepsilon_0 } $ (for a certain $ \varepsilon_0 $), then it can be represented as
\begin{equation*}
\zeta = ( 1 + \omega )e^{ i \varphi / 2 } + \rho e^{ i \theta }, \quad \rho \leqslant \varepsilon_0.
\end{equation*}

Let us estimate the ratio $ \dfrac{ r( \zeta ) }{ S'_{ \zeta }( \zeta, u ) } $ in $ D_{ \varepsilon_0 } $. If $ \zeta_0 $ belongs to the $ 2 \varepsilon_0 $--neighborhood of $ \zeta_1 $, then all $ \zeta \in Z_0^+ \bigcap D_{ \varepsilon_0 } $ belong to the $ 3 \varepsilon_0 $--neighborhood of $ \zeta_1 $. The following estimates hold:
\begin{equation}
\label{difference_estimates}
| \zeta + \zeta_1 | \geqslant 2 - 3 \varepsilon_0, \quad | \zeta + \zeta_0 | > 2 - \varepsilon_0, \quad | \zeta + \zeta_2 | \geqslant 1 + \varepsilon_0, \quad \xi \in Z_0^+ \bigcap D_{ \varepsilon_0 }.
\end{equation}

Further, we note that for any $ N \in \mathbb{N} $ there exists a function $ \tilde r( \zeta ) $ that can be represented in the form (\ref{r_b_connection}) with a certain $ \tilde b( \zeta ) $ satisfying properties (\ref{b_continuity})--(\ref{b_derivs}), such that $ | r( \zeta ) | \leqslant | \zeta - \zeta_1 |^N | \tilde r( \zeta ) | $ for $ \zeta $ belonging to the $ 3 \varepsilon_0 $--neighborhood of $ \zeta_1 $. This and (\ref{difference_estimates}) imply that
\begin{equation}
\label{ratio_estimate}
\left| \frac{ r( \zeta ) }{ S'_{ \zeta }( \zeta, u ) } \right| \leqslant \const \frac{ | \tilde r( \zeta ) | | \zeta |^4 }{ | \zeta - \zeta_2 | | \zeta - \zeta_0 | }.
\end{equation}

A similar reasoning holds for the case when $ \zeta_0 $ belongs to the $ 2 \varepsilon_0 $--neighborhood of $ - \zeta_1 $. Now if $ \zeta_0 $ does not belong to the $ 2 \varepsilon_0 $--neighborhood of $ \zeta_1 $ and $ - \zeta_1 $, then $ | \zeta - \zeta_1 | \geqslant \varepsilon_0 $ and $ | \zeta + \zeta_1 | \geqslant \varepsilon_0 $ for all $ \zeta \in Z_0^+ \bigcap D_{ \varepsilon_0 } $. Two last estimates of (\ref{difference_estimates}) hold and thus (\ref{ratio_estimate}) holds with $ \tilde r( \zeta ) = r( \zeta ) $.

The difference $ | \zeta - \zeta_2 | $ can be estimated
\begin{equation}
\label{zeta_omega}
| \zeta - \zeta_2 | \geqslant \frac{ 1 }{ 2 } | \zeta_0 - \zeta_2 | = \frac{ \omega ( 2 + \omega ) }{ 2 ( 1 + \omega ) }, \quad \zeta \in Z_0^+.
\end{equation}
In order to get rid of this member in the denominator, let us represent $ \tilde r( \zeta ) $ by the Taylor formula in the neighborhood of $ ( 1 + \omega ) e^{ i \varphi / 2 } $:
\begin{equation*}
\tilde r( \zeta ) = \tilde r( e^{ i \varphi / 2 } + \omega e^{ i \varphi / 2 } ) + \tilde r'( e^{ i \varphi / 2 } + \omega e^{ i \varphi / 2 } + s e^{ i \theta } ) \rho, \quad s = \lambda \rho \text{ for some } \lambda \in [ 0, 1 ],
\end{equation*}
where $ ' $ denotes the derivative with respect to $ s $ and where $ \lambda $ depends, in particular, on $ \rho $.

For an arbitrary value of $ \omega $ the following estimates hold:
\begin{align}
\label{tilde_r_estimate}
& | \tilde r( e^{ i \varphi / 2 } + \omega e^{ i \varphi / 2 } ) | \leqslant \const | \omega |^N, \\
\label{tilde_r_prime_esimate}
& | \tilde r'( e^{ i \varphi / 2 } + \omega e^{ i \varphi / 2 } ) | \leqslant \const | \omega |^N.
\end{align}
This finally yields
\begin{multline*}
\frac{ | \tilde r( \zeta ) | | \zeta |^4 }{ | \zeta - \zeta_2 | | \zeta - \zeta_0 | } \leqslant \const \left( \frac{ | \tilde r( e^{ i \varphi / 2 } + \omega e^{ i \varphi / 2 } ) | }{ \frac{ \omega ( 2 + \omega ) }{ 2 ( 1 + \omega ) } | \zeta - \zeta_0 | } + \frac{ | \tilde r'( e^{ i \varphi / 2 } + \omega e^{ i \varphi / 2 } + \lambda \rho e^{ i \theta } ) | \rho }{ \rho | \zeta - \zeta_0 | } \right) \leqslant \\
\leqslant \frac{ \const }{ | \zeta - \zeta_0 | }, \quad \zeta \in Z_0^+ \bigcap D_{ \varepsilon_0 }.
\end{multline*}

Now we are ready to estimate $ I_1 $:
\begin{equation*}
\int\limits_{ Z_0^+ \bigcap \partial D_{ \varepsilon } } \left| \frac{ r( \zeta ) \exp( i t S( u, \zeta ) ) }{ S'_{ \zeta }( u, \zeta ) } \right| d \bar \zeta \leqslant \const \int\limits_{ Z_0^+ \bigcap \partial D_{ \varepsilon } } \frac{ d \bar \zeta }{ | \zeta - \zeta_0 | } = O\left( 1 \right) \quad \text{as} \quad \varepsilon \to 0.
\end{equation*}

For $ I_2 $ we note that the estimate
\begin{equation*}
\left| \frac{ r'_{ \zeta }( \zeta ) }{ S'_{ \zeta }( u, \zeta ) } \right| \leqslant \frac{ \const }{ | \zeta - \zeta_0 | }, \quad \zeta \in Z_0^+ \bigcap D_{ \varepsilon_0 },
\end{equation*}
can be obtained using the same reasoning as for the ratio $ \dfrac{ r( \zeta ) }{ S'_{ \zeta }( u, \zeta ) } $. Thus
\begin{multline*}
\iint\limits_{ Z_0^+ \backslash D_{ \varepsilon } } \left| \frac{ r'_{ \zeta }( \zeta ) \exp( i t S( u, \zeta ) ) }{ S'_{ \zeta }( u, \zeta ) } \right| d \Re \zeta d \Im \zeta  \leqslant \const \Biggl( \iint\limits_{ Z_0^+ \backslash D_{ \varepsilon_0 } } | r'_{ \zeta }( \zeta ) | | \zeta |^4 d \Re \zeta d \Im \zeta + \\
+ \iint\limits_{ Z_0^+ \bigcap ( D_{ \varepsilon_0 } \backslash D_{ \varepsilon } ) } \frac{ d \Re \zeta d \Im \zeta }{ | \zeta - \zeta_0 | } \Biggr) = O( 1 ) \quad \text{as} \quad \varepsilon \to 0.
\end{multline*}

In order to estimate $ \dfrac{ r( \zeta ) }{ ( S'_{ \zeta }( u, \zeta ) )^2 } $ in $ D_{ \varepsilon_0 } $ we take the members in the Taylor formula for $ \tilde r( \zeta ) $ up to the second order:
\begin{equation}
\label{taylor_second_order}
\tilde r( \zeta ) = \tilde r( e^{ i \varphi / 2 } + \omega e^{ i \varphi / 2 } ) + \tilde r'( e^{ i \varphi / 2 } + \omega e^{ i \varphi / 2 } ) \rho + \frac{ 1 }{ 2 }\tilde r''( e^{ i \varphi / 2 } + \omega e^{ i \varphi / 2 } + \lambda \rho e^{ i \theta } ) \rho^2.
\end{equation}
From formulas (\ref{ratio_estimate})--(\ref{taylor_second_order}) it follows that
\begin{multline*}
\left| \frac{ r( \zeta ) }{ ( S'_{ \zeta }( u, \zeta ) )^2 } \right| \leqslant \frac{ | \tilde r( \zeta ) | | \zeta |^8 }{ | \zeta - \zeta_2 |^2 | \zeta - \zeta_0 |^2 } \leqslant \const \Biggl( \frac{ | \tilde r( e^{ i \varphi / 2 } + \omega e^{ i \varphi / 2 } ) | }{ \left( \frac{ \omega ( 2 + \omega ) }{ 2 ( 1 + \omega ) } \right)^2 | \zeta - \zeta_0 |^2 } + \\
+ \frac{ | \tilde r'( e^{ i \varphi / 2 } + \omega e^{ i \varphi / 2 } ) | \, \rho }{ \left( \frac{ \omega ( 2 + \omega ) }{ 2 ( 1 + \omega ) } \right)^2 | \zeta - \zeta_0 |^2 } + \frac{ | \tilde r''( e^{ i \varphi / 2 } + \omega e^{ i \varphi / 2 } + \lambda \rho e^{ i \theta } ) | \rho^2 }{ 2 \rho^2 | \zeta - \zeta_0 |^2 } \Biggr) \leqslant \\
\leqslant \frac{ \const }{ | \zeta - \zeta_0 |^2 }, \quad \zeta \in Z_0^+ \bigcap D_{ \varepsilon_0 }.
\end{multline*}
Thus
\begin{multline*}
\iint\limits_{ Z_0^+ \bigcap ( D_{ \varepsilon_0 } \backslash D_{ \varepsilon } ) } \left| \frac{ r( \zeta ) \exp( i t S( u, \zeta ) ) S''_{ \zeta \zeta }( u, \zeta ) }{ ( S'_{ \zeta }( u, \zeta ) )^2 } \right| d \Re \zeta d \Im \zeta \leqslant \\
\leqslant \const \int\limits_{ \varepsilon }^{ \varepsilon_0 } \frac{ d \rho }{ \rho } = \const \ln \frac{ \varepsilon_0 }{ \varepsilon }.
\end{multline*}
Setting $ \varepsilon = \frac{ 1 }{ | t | } $ yields $ I_3 = O( \ln( | t | ) ) $, as $ t \to \infty $.

\item[(B)] Now let $ \dfrac{ 1 }{ 1 + \omega } \leqslant 2 \varepsilon_0 $.

If $ \zeta \in Z_0^+ \bigcap D_{ \varepsilon_0 } $, then the following estimates hold
\begin{equation*}
| \zeta \pm \zeta_1 | \geqslant \frac{ 1 }{ 2 \varepsilon_0 } - 1 - \varepsilon_0, \quad | \zeta \pm \zeta_2 | \geqslant \frac{ 1 }{ 2 \varepsilon_0 } - 3 \varepsilon_0, \quad | \zeta + \zeta_0 | > 2 - \varepsilon_0.
\end{equation*}

Consequently,
\begin{equation*}
| S'_{ \zeta }( u, \zeta ) | \geqslant \frac{ \const }{ | \zeta |^4 } | \zeta - \zeta_0 |, \quad \zeta \in Z_0^+ \bigcap D_{ \varepsilon_0 }.
\end{equation*}
and the part of the integral $ I_{ext} $ over $ Z_0^+ $ for this case can be estimated, proceeding as in the previous section, as $ O\left( \frac{ \ln( | t | ) }{ | t | } \right) $, $ t \to \infty $.

If $ \zeta \in Z_2^+ \bigcap D_{ \varepsilon_0 } $, then the following estimates hold
\begin{equation*}
| \zeta \pm \zeta_1 | \geqslant 1 - 3 \varepsilon_0, \quad | \zeta \pm \zeta_0 | \geqslant 1 - 3 \varepsilon_0,
\end{equation*}
and thus
\begin{equation*}
| S'_{ \zeta }( u, \zeta ) | \geqslant \frac{ \const }{ | \zeta |^4 } | \zeta - \zeta_2 | | \zeta + \zeta_2 |.
\end{equation*}

We can estimate $ | \zeta + \zeta_2 | \geqslant | \zeta_2 | = \dfrac{ 1 }{ 1 + \omega } $. Now let us expand $ r( \zeta ) $ into Taylor formula in the neighborhood of $ \frac{ 1 }{ 1 + \omega } e^{ i \varphi / 2 } $:
\begin{equation*}
r( \zeta ) = r\left( \frac{ 1 }{ 1 + \omega } e^{ i \varphi / 2 } \right) + r'\left( \frac{ 1 }{ 1 + \omega } e^{ i \varphi / 2 } + \lambda \rho e^{ i \theta } \right) \rho, \quad \lambda \in [ 0, 1 ].
\end{equation*}
For an arbitrary value of $ \omega > 0 $ (satisfying $ \frac{ 1 }{ 1 + \omega } \leqslant 2 \varepsilon_0 $) the following estimate holds:
\begin{equation*}
\left| r\left( \frac{ 1 }{ 1 + \omega } e^{ i \varphi / 2 }  \right) \right| \leqslant \const \left| \frac{ 1 }{ 1 + \omega } \right|^N.
\end{equation*}

This yields
\begin{equation*}
\left| \frac{ r( \zeta ) }{ S'_{ \zeta }( u, \zeta ) } \right| \leqslant  \const \left( \frac{ | r( \frac{ 1 }{ 1 + \omega } e^{ i \varphi / 2 } ) | }{ \frac{ 1 }{ 1 + \omega } | \zeta - \zeta_2 | } + \frac{ | r'( \frac{ 1 }{ 1 + \omega } e^{ i \varphi / 2 } + \lambda \rho e^{ i \theta } ) | \rho }{ \rho | \zeta - \zeta_2 | } \right) \leqslant \frac{ \const }{ | \zeta - \zeta_2 | }.
\end{equation*}
In the same manner
\begin{multline*}
\left| \frac{ r( \zeta ) }{ ( S'_{ \zeta }( u, \zeta ) )^2 } \right| \leqslant \frac{ | r( \zeta ) | | \zeta |^8 }{ | \zeta + \zeta_2 |^2 | \zeta - \zeta_2 |^2 } \leqslant \const \Biggl( \frac{ | r( \frac{ 1 }{ 1 + \omega } e^{ i \varphi / 2 } ) | }{ \left( \frac{ 1 }{ 1 + \omega } \right)^2 | \zeta - \zeta_2 |^2 } + \\
+ \frac{ | r'( \frac{ 1 }{ 1 + \omega } e^{ i \varphi / 2 } ) | \, \rho }{ \left( \frac{ 1 }{ 1 + \omega } \right)^2 | \zeta - \zeta_2 |^2 } + \frac{ | r''( \frac{ 1 }{ 1 + \omega } e^{ i \varphi / 2 } + \lambda \rho e^{ i \theta } ) | \rho^2 }{ 2 \rho^2 | \zeta - \zeta_2 |^2 } \Biggr) \leqslant \frac{ \const }{ | \zeta - \zeta_2 |^2 }.
\end{multline*}
Following further the reasoning from the case (A) we obtain that
\begin{equation*}
I_{ ext } = O\left( \frac{ \ln( | t | ) }{ | t | } \right), \quad \text{as} \quad t \to \infty
\end{equation*}
uniformly on $ u \in \mathbb{C} $. \qedhere
\end{itemize}
\end{itemize}
\end{proof}

\end{document}